\documentclass[leqno,12pt]{article} 
\usepackage{amsmath, amssymb}
\usepackage{amsthm} 
\theoremstyle{plain} 
\newtheorem{theorem}{\indent\sc Theorem}[section]
\newtheorem{lemma}[theorem]{\indent\sc Lemma}
\newtheorem{corollary}[theorem]{\indent\sc Corollary}
\newtheorem{proposition}[theorem]{\indent\sc Proposition}

\theoremstyle{definition} 

\newtheorem{remark}[theorem]{\indent\sc Remark}
\newtheorem{example}[theorem]{\indent\sc Example}

\newcommand\on{\operatorname}
\renewcommand\div{\on{div}}
\newcommand\Hess{\on{Hess}}
\newcommand\R{\on{Ric}}

\newcommand\sign{\on{sign}}

\newcommand\sech{\on{sech}}

\title{Constructing Ricci vector fields \\ on $\mathbb R^2$ with a diagonal metric}
\author{Adara M. Blaga
}

\date{}
\begin{document}

\maketitle

\markboth{{\small\it {\hspace{1cm} Constructing Ricci vector fields on $\mathbb R^2$ with a diagonal metric}}}{\small\it{Constructing Ricci vector fields on $\mathbb R^2$ with a diagonal metric\hspace{1cm}}}


\begin{abstract}
We put into light Ricci vector fields on $\mathbb R^2$ endowed with a diagonal metric.
\end{abstract}

\section{Preliminaries}

Ricci vector fields constitute a special class of vector fields on a Riemannian manifold, whose properties have been studied in \cite{alde, aldevi, adara, ri}. We recall that a vector field $V$ is called an $f$-\textit{Ricci vector field} \cite{ri} if $\nabla_XV=f\cdot QX$ for all smooth vector fields $X$ on $\mathbb R^2$, where $f$ is a smooth real-valued function on $\mathbb R^2$, $Q$ is the Ricci operator defined by $g(QX_1,X_2):=\R(X_1,X_2)$ for $X_1,X_2$ smooth vector fields on $\mathbb R^2$, and $\R$ is the Ricci curvature tensor field of $(\mathbb R^2,g)$. If $f=1$, we have a particular class of Ricci vector fields, which we shall denote by $\mathfrak R$, which have some special properties. Let us remark that, for any $V\in \mathfrak R$ and any $X_1,X_2$ smooth vector fields on $\mathbb R^2$, we have
\begin{align*}
R(X_1,X_2)V&=\nabla_{X_1}\nabla_{X_2}V-\nabla_{X_2}\nabla_{X_1}V-\nabla_{[X_1,X_2]}V\\
&=\nabla_{X_1}(QX_2)-\nabla_{X_2}(QX_1)-Q(\nabla_{X_1}X_2-\nabla_{X_2}X_1)\\
&=(\nabla_{X_1}Q)X_2-(\nabla_{X_2}Q)X_1,\\
\R(V,V)&=g(QV,V)=g(\nabla_VV,V)=\frac{1}{2}V(\Vert V\Vert^2),\\
r&=\sum_{i=1}^2\R(E_i,E_i)=\sum_{i=1}^2g(QE_i,E_i)=\sum_{i=1}^2g(\nabla_{E_i}V,E_i)=\div(V),
\end{align*}
where $R$ is the Riemannian curvature tensor field and $r$ is the scalar curvature of $(M,g)$. On the other hand, any $V\in \mathfrak R$ is a closed vector field, $V_1-V_2$ is a parallel vector field, $\displaystyle\frac{V_1+V_2}{2}\in \mathfrak R$, and $V_1(\Vert V_2\Vert^2)=V_2(\Vert V_1\Vert^2)$ for any $V_1,V_2\in \mathfrak R$. Indeed, if $\eta$ is the dual $1$-form of $V$, then, for any smooth vector fields $X,X_1,X_2$ on $\mathbb R^2$, we have
\begin{align*}
(d\eta)(X_1,X_2)&=X_1(\eta(X_2))-X_2(\eta(X_1))-\eta([X_1,X_2])\\
&=X_1(g(V,X_2))-X_2(g(V,X_1))-g(V,\nabla_{X_1}X_2-\nabla_{X_2}X_1)\\
&=g(\nabla_{X_1}V,X_2)-g(\nabla_{X_2}V,X_1)\\
&=g(QX_1,X_2)-g(QX_2,X_1)=0,\\
\nabla_{X}(V_1-V_2)&=0,\ \
\nabla_{X}\left(\displaystyle \frac{V_1+V_2}{2}\right)=QX.
\end{align*}

Moreover, if $V$ is a gradient vector field, $V=\nabla f$, then, $\Hess(f)=\R$, $\Delta(f)=r$, and $\displaystyle \frac{1}{2}\pounds_{\nabla(-f)}g+\R=0$, i.e., $(\mathbb R^2,g,\nabla(-f),0)$ is a gradient steady Ricci soliton (see \cite{hamilton} for its definition). We mention that, in \cite{blvi}, we put into relation statistical structures with Ricci and Hessian metrics to gradient Ricci solitons (see also \cite{b,BC1}). On the other hand, almost $\eta$-Ricci solitons with a diagonal metric have been recently considered in \cite{bl22}.

\bigskip

In the present paper, we aim to construct Ricci vector fields $V$ with respect to a diagonal Riemannian metric $g$, which satisfy $\nabla V=Q$, equivalent to, $$g(\nabla_{E_i} V,E_j)=\R(E_i,E_j)$$ for any $i,j\in \{1,2\}$. A similar study concerning the Killing vector fields on $\mathbb R^2$ with the same metric has been done by the author in \cite{bl2e}.

\section{Ricci vector fields w.r.t. these metrics}

Let
${g}=\displaystyle \frac{1}{f_1^2}dx^1\otimes dx^1+\displaystyle \frac{1}{f_2^2}dx^2\otimes dx^2$, where $f_1,f_2$ are smooth real-valued functions nowhere zero on $\mathbb R^2$ and $x^1,x^2$ are the standard coordinates in $\mathbb R^2$, when the two functions depend on one of the variables,
and let
$$\Big\{E_1:=f_1\frac{\partial}{\partial x^1}, \ \ E_2:=f_2\frac{\partial}{\partial x^2}\Big\}$$
be a local orthonormal frame. We will denote as follows:
$$\frac{f_2}{f_1}\cdot\frac{\partial f_1}{\partial x^2}=:h_{12}, \ \ \frac{f_1}{f_2}\cdot\frac{\partial f_2}{\partial x^1}=:h_{21}.$$
On the base vector fields, the Lie bracket is given by:
$$[E_1,E_2]=-h_{12}E_1+h_{21}E_2=-[E_2,E_1],$$
and the Levi-Civita connection $\nabla$ of $g$ is given by (see \cite{blagalatcu}):
$$\nabla_{E_1}E_1=h_{12}E_2, \ \ \nabla_{E_2}E_2=h_{21}E_1, \ \ \nabla_{E_1}E_2=-h_{12}E_1, \ \ \nabla_{E_2}E_1=-h_{21}E_2.$$

Whenever a function $f$ on $\mathbb R^2$ depends only on one of its variables, we will write in its argument only that variable in order to underline this fact, for example, $f(x^i)$.

\bigskip

Let $V=\sum_{k=1}^2V^kE_k$ with $V^1,V^2$ two smooth functions on $\mathbb R^2$. Then,
\begin{align*}
g(\nabla_{E_i}V,E_j)&=E_i(V^j)+V^1 g(\nabla_{E_i}E_1,E_j)+V^2 g(\nabla_{E_i}E_2,E_j)
\end{align*}
for any $i,j\in \{1,2\}$, which is equivalent to
\begin{equation}\label{s1}
\left\{
    \begin{aligned}
&g(\nabla_{E_1}V,E_1)=E_1(V^1)-h_{12}V^2\\
&g(\nabla_{E_2}V,E_2)=E_2(V^2)-h_{21}V^1\\
&g(\nabla_{E_1}V,E_2)=E_1(V^2)+h_{12}V^1\\
&g(\nabla_{E_2}V,E_1)=E_2(V^1)+h_{21}V^2
    \end{aligned}
  \right. \ .
\end{equation}

We know that, on the base vector fields $E_1=f_1\displaystyle\frac{\partial}{\partial x^1}$, $E_2=f_2 \displaystyle\frac{\partial}{\partial x^2}$, the Ricci curvature tensor field of the metric ${g}=\displaystyle \frac{1}{f_1^2}dx^1\otimes dx^1+\displaystyle \frac{1}{f_2^2}dx^2\otimes dx^2$ is given by \cite{blagalatcu}:
$$\R(E_1,E_1)=E_1(h_{21})+E_2(h_{12})-h_{21}^2-h_{12}^2=\R(E_2,E_2),$$$$\R(E_1,E_2)=0,$$
and we can state
\begin{lemma}\label{lps}
The vector field $V=\sum_{k=1}^2V^kE_k$ with $V^1,V^2$ two smooth functions on $\mathbb R^2$, is a Ricci vector field satisfying $\nabla V=Q$ if and only if
 \begin{equation}\label{s1s}
\left\{
    \begin{aligned}
&f_1\displaystyle{\frac{\partial V^1}{\partial x^1}}-\displaystyle{\frac{f_2}{f_1}}\cdot\displaystyle{\frac{\partial f_1}{\partial x^2}}V^2=f_1\displaystyle{\frac{\partial }{\partial x^1}}\left(\displaystyle{\frac{f_1}{f_2}}\cdot\displaystyle{\frac{\partial f_2}{\partial x^1}}\right)-\left(\displaystyle{\frac{f_1}{f_2}}\cdot\displaystyle{\frac{\partial f_2}{\partial x^1}}\right)^2+f_2\displaystyle{\frac{\partial }{\partial x^2}}\left(\displaystyle{\frac{f_2}{f_1}}\cdot\displaystyle{\frac{\partial f_1}{\partial x^2}}\right)-\left(\displaystyle{\frac{f_2}{f_1}}\cdot\displaystyle{\frac{\partial f_1}{\partial x^2}}\right)^2\\
&f_1\displaystyle{\frac{\partial V^1}{\partial x^1}}+\displaystyle{\frac{f_1}{f_2}}\cdot\displaystyle{\frac{\partial f_2}{\partial x^1}}V^1=f_2\displaystyle{\frac{\partial V^2}{\partial x^2}}+\displaystyle{\frac{f_2}{f_1}}\cdot\displaystyle{\frac{\partial f_1}{\partial x^2}}V^2\\
&f_1\displaystyle{\frac{\partial V^2}{\partial x^1}}=-\displaystyle{\frac{f_2}{f_1}}\cdot\displaystyle{\frac{\partial f_1}{\partial x^2}}V^1\\
&f_2\displaystyle{\frac{\partial V^1}{\partial x^2}}=-\displaystyle{\frac{f_1}{f_2}}\cdot\displaystyle{\frac{\partial f_2}{\partial x^1}}V^2
    \end{aligned}
  \right. \ .
\end{equation}
\end{lemma}

Let us firstly determine the conditions for the base vector fields $E_1$ and $E_2$ to be Ricci vector fields.

\begin{proposition}
The vector field $E_i=f_i\frac{\displaystyle \partial}{\displaystyle \partial x^i}$, for $i\in \{1,2\}$, is a Ricci vector field satisfying $\nabla E_i=Q$ if and only if $f_1=f_1(x^1)$ and $f_2=f_2(x^2)$.
\end{proposition}
\begin{proof}
The 2nd and 3rd equation of (\ref{s1s}) implies
\begin{equation*}
\left\{
    \begin{aligned}
&\frac{\displaystyle \partial f_2}{\displaystyle \partial x^1}=0\\
&\frac{\displaystyle \partial f_1}{\displaystyle \partial x^2}=0
    \end{aligned}
  \right. \ ,
\end{equation*}
hence, $f_1=f_1(x^1)$ and $f_2=f_2(x^2)$. Conversely, if $f_1$ depends only on $x^1$ and $f_2$ depends only on $x^2$, then, the system (\ref{s1s}) is verified.
\end{proof}

\begin{example}
$E_1$ and $E_2$ are Ricci vector fields on
$$\left(\mathbb R^2, \ {g}=e^{-2x^1}dx^1\otimes dx^1+e^{-2x^2}dx^2\otimes dx^2\right).$$
\end{example}

\begin{theorem}\label{t1}
If $f_i=f_i(x^i)$ for $i\in \{1,2\}$, then, $V=\sum_{k=1}^2V^kE_k$ is a Ricci vector field satisfying $\nabla V=Q$ if and only if $V^k=c_k\in \mathbb R$, $k\in \{1,2\}$. In particular, if $f_1=k_1\in \mathbb R\setminus \{0\}$ and $f_2=k_2\in \mathbb R\setminus \{0\}$, the conclusion holds true.
\end{theorem}
\begin{proof}
In this case, \eqref{s1s} becomes
$$\displaystyle{\frac{\partial V^i}{\partial x^j}}=0, \ \ \textrm{for} \ \ i,j\in\{1,2\},$$
and we get the conclusion.
\end{proof}

\begin{example}
Let $f_1(x^1)=e^{x^1}$ and $f_2(x^2)=e^{x^2}$. Then,
$$V=e^{x^1}\displaystyle\frac{\partial }{\partial x^1}+e^{x^2}\displaystyle\frac{\partial }{\partial x^2}$$
is a Ricci vector field on $\left(\mathbb R^2, \ {g}=e^{-2x^1}dx^1\otimes dx^1+e^{-2x^2}dx^2\otimes dx^2\right)$.
\end{example}

\begin{theorem}
If $f_i=f_i(x^1)$ for $i\in \{1,2\}$, and $f_2'(x^1)\neq 0$ for all $x^1\in \mathbb R$, then, $V=\sum_{k=1}^2V^kE_k$ is a Ricci vector field satisfying $\nabla V=Q$ if and only if
\begin{equation*}
\left\{
    \begin{aligned}
&V^1(x^1)=\frac{c}{f_2(x^1)}\\
&V^2=0
\end{aligned}
\right.  \ \ \textrm{and} \ \ f_1=\displaystyle\frac{kf_2^2+c}{2f_2'}
\end{equation*}
or
\begin{equation*}
\left\{
    \begin{aligned}
&V^1(x^2)=\sign(c)\left[c_2\cos(|c|x^2)-c_1\sin(|c|x^2)\right] \\ &V^2(x^2)=c_1\cos(|c|x^2)+c_2\sin(|c|x^2)
\end{aligned}
\right.  \ \ \textrm{and} \ \ f_1=c\displaystyle\frac{f_2^2}{f_2'},
\end{equation*}
with $k,c\in \mathbb R\setminus\{0\}$, $c_1,c_2\in \mathbb R$.
\end{theorem}
\begin{proof}
In this case, \eqref{s1s} becomes
\begin{equation}\label{s1sa}
\left\{
    \begin{aligned}
&\displaystyle{\frac{\partial V^1}{\partial x^1}}=f_1'\displaystyle{\frac{f_2'}{f_2}}+f_1\left[\left(\displaystyle{\frac{f_2'}{f_2}}\right)'-\left(\displaystyle{\frac{f_2'}{f_2}}\right)^2\right]\\
&f_1\left(\displaystyle{\frac{\partial V^1}{\partial x^1}}+\displaystyle{\frac{f_2'}{f_2}}V^1\right)=f_2\displaystyle{\frac{\partial V^2}{\partial x^2}}\\
&\displaystyle{\frac{\partial V^2}{\partial x^1}}=0\\
&\displaystyle{\frac{\partial V^1}{\partial x^2}}=-f_1\displaystyle{\frac{f_2'}{f_2^2}}V^2
    \end{aligned}
  \right. \ .
\end{equation}
From the 3rd equation, we deduce that
$$V^2=V^2(x^2)$$
and from the 1st equation, we get
$$V^1(x^1,x^2)=F(x^1)+F_1(x^2),$$
where $F'=f_1'\displaystyle{\frac{f_2'}{f_2}}+f_1\left[\left(\displaystyle{\frac{f_2'}{f_2}}\right)'-\left(\displaystyle{\frac{f_2'}{f_2}}\right)^2\right]$.
By differentiating $V^1$ with respect to $x^2$ and considering the 4th equation of the system, we find that
$$F_1'(x^2)=-\left(f_1\displaystyle{\frac{f_2'}{f_2^2}}\right)(x^1)\ V^2(x^2).$$
Then, the 2nd equation of the system becomes
$$f_1(x^1)\left[F'(x^1)+\left(\displaystyle{\frac{f_2'}{f_2}}\right)(x^1)\ [F(x^1)+F_1(x^2)]\right]=f_2(x^1)(V^2)'(x^2),$$
which is equivalent to
$$\left(\frac{f_1}{f_2}\right)(x^1)\left[F'(x^1)+\left(\frac{f_2'}{f_2}\right)(x^1)\ F(x^1)\right]=(V^2)'(x^2)-\left(f_1\displaystyle{\frac{f_2'}{f_2^2}}\right)(x^1)\ F_1(x^2),$$
which, by differentiation with respect to $x^2$ implies
$$(V^2)''(x^2)=\left(f_1\displaystyle{\frac{f_2'}{f_2^2}}\right)(x^1)\ F_1'(x^2)=-\left(f_1\displaystyle{\frac{f_2'}{f_2^2}}\right)^2(x^1)\ V^2(x^2).$$
Since $f_2'(x^1)\neq 0$ for all $x^1\in \mathbb R$, we get $V^2=0$ or $V^2(x^2)=c_1\cos(|c|x^2)+c_2\sin(|c|x^2)$, $c=f_1\displaystyle{\frac{f_2'}{f_2^2}}, c_1,c_2\in \mathbb R$.

\hspace{0.5cm} (1) If $V^2=0$, then $F_1'=0$, hence, $F_1=c_1\in \mathbb R$ and $V^1(x^1)=F(x^1)+c_1$. It follows that
$$\left(\frac{f_1}{f_2}\right)(x^1)\left[F'(x^1)+\left(\frac{f_2'}{f_2}F\right)(x^1)\right]=-\left(f_1\displaystyle{\frac{f_2'}{f_2^2}}\right)(x^1)\ c_1,$$
which implies that
$$(f_2F)'(x^1)=-c_1f_2'(x^1),$$
therefore,
$$F(x^1)=-c_1+\frac{c_2}{f_2(x^1)}, \ c_2\in \mathbb R.$$
Then $V^1(x^1)=\displaystyle\frac{c_2}{f_2(x^1)}$, and the 1st equation of the system implies
$$f_1'f_2+\displaystyle{\frac{f_1f_2^2}{f_2'}}\left[\left(\displaystyle{\frac{f_2'}{f_2}}\right)'-\left(\displaystyle{\frac{f_2'}{f_2}}\right)^2\right]=-c_2,$$
which is equivalent to
$$\displaystyle{\frac{f_2}{f_2'}}(f_1f_2')'-2(f_1f_2')=-c_2.$$
If we denote by $f:=f_1f_2'$, the previous relation can be written as
$$f'=\displaystyle{\frac{f_2'}{f_2}}(2f-c_2).$$
Now, if $f=\displaystyle{\frac{c_2}{2}}$, then $f_1=\displaystyle{\frac{c_2}{2f_2'}}$. If there exists $x_0^1\in \mathbb R$ such that $f(x_0^1)\neq 0$, then, $f(x^1)\neq 0$ for all $x^1\in I:=(x_0^1-\varepsilon, x_0^1+\varepsilon)\subseteq \mathbb R$ for an $\varepsilon>0$, and $\displaystyle{\frac{f'}{2f-c_2}}=\displaystyle{\frac{f_2'}{f_2}}$ on $I$, which, by integration, gives $f_1f_2'=f=\displaystyle{\frac{kf_2^2+c_2}{2}}$ on $I$, where $k\in \mathbb R\setminus \{0\}$. We get $f_1=\displaystyle{\frac{kf_2^2+c_2}{2f_2'}}$ on $I$, from where, since $f_2$ is nowhere $0$ on $\mathbb R$, follows that $f_1$ has this expression on $\mathbb R$.

\hspace{0.5cm} (2) If $V^2(x^2)=c_1\cos(|c|x^2)+c_2\sin(|c|x^2)$ with $f_1\displaystyle{\frac{f_2'}{f_2^2}}=c\in \mathbb R$, we get $f_2(x^1)=\nolinebreak -\displaystyle\frac{1}{cF_2(x^1)}$, where $F_2'=\displaystyle\frac{1}{f_1}$.
Then
$$F_1'(x^2)=-\left(f_1\displaystyle{\frac{f_2'}{f_2^2}}\right)(x^1)\ V^2(x^2)=-cc_1\cos(|c|x^2)-cc_2\sin(|c|x^2),$$
which infers
$$F_1(x^2)=-\frac{cc_1}{|c|}\sin(|c|x^2)+\frac{cc_2}{|c|}\cos(|c|x^2)+c_3, \ \ c_3\in \mathbb R,$$
equivalent to
$$F_1(x^2)=\sign(c)\left[c_2\cos(|c|x^2)-c_1\sin(|c|x^2)\right]+c_3.$$
Then,
$$V^1(x^1,x^2)=\tilde F(x^1)+\sign(c)\left[c_2\cos(|c|x^2)-c_1\sin(|c|x^2)\right],$$
where $\tilde F'=F'$. Replacing now $V^1$ and $V^2$ in the 2nd equation, we get
$$\tilde F'=-\frac{f_2'}{f_2}\tilde F.$$
If there exists a real open interval $I$ such that $\tilde F(x^1)\neq 0$ for all $x^1\in I$, then $\tilde F=\displaystyle\frac{k}{f_2}$ with $k\in \mathbb R\setminus \{0\}$ on $I$, and taking into account that
$$-k\displaystyle{\frac{f_2'}{f_2^2}}=\tilde F'=F'=f_1'\displaystyle{\frac{f_2'}{f_2}}+f_1\left[\left(\displaystyle{\frac{f_2'}{f_2}}\right)'-\left(\displaystyle{\frac{f_2'}{f_2}}\right)^2\right]$$
on $I$, and that $f_1=c\displaystyle\frac{f_2^2}{f_2'}$, we infer $k=0$, hence we obtain a contradiction. Therefore, $\tilde F=0$, and we get
$$V^1(x^2)=\sign(c)\left[c_2\cos(|c|x^2)-c_1\sin(|c|x^2)\right]. \qedhere$$
\end{proof}

\begin{example}
Let $f_1(x^1)=\cosh(x^1)$ and $f_2(x^1)=e^{x^1}$. Then,
$$V=e^{-x^1}\cosh(x^1)\displaystyle\frac{\partial }{\partial x^1}$$
is a Ricci vector field on $\left(\mathbb R^2, \ {g}=\sech^2(x^1)dx^1\otimes dx^1+e^{-2x^1}dx^2\otimes dx^2\right)$.
\end{example}

\begin{example}
Let $f_1(x^1)=k_1e^{x^1}$, $k_1\in \mathbb R\setminus \{0\}$, and $f_2(x^1)=k_2e^{x^1}$, $k_2\in \mathbb R\setminus \{0\}$, with $k_1k_2>0$. Then,
$$V=k_1^2e^{x^1}\left[\cos \left(\frac{k_1}{k_2}x^2\right)-\sin\left(\frac{k_1}{k_2}x^2\right)\right]\displaystyle\frac{\partial }{\partial x^1}+k_2^2e^{x^1}\left[\cos \left(\frac{k_1}{k_2}x^2\right)+\sin\left(\frac{k_1}{k_2}x^2\right)\right]\displaystyle\frac{\partial }{\partial x^2}$$
is a Ricci vector field on $\left(\mathbb R^2, \ {g}=k_1^{-2}e^{-2x^1}dx^1\otimes dx^1+k_2^{-2}e^{-2x^1}dx^2\otimes dx^2\right)$.
\end{example}

\begin{remark}
Examples of nonconstant functions $f_1$ and $f_2$ which satisfy the condition $$f_1'f_2-2f_1f_2'+f_1f_2\displaystyle\frac{f_2''}{f_2'}=\textrm{nonzero} \ \textrm{constant}=:c$$ are
$$f_1(x^1)=-\frac{c}{ka}\cosh(ax^1), \ \ f_2(x^1)=ke^{ax^1}, \ \ k,a\in \mathbb R\setminus \{0\},$$
$$f_1(x^1)=\frac{c}{ka}\sinh(ax^1), \ \ f_2(x^1)=ke^{ax^1}, \ \ k,a\in \mathbb R\setminus \{0\},$$
and an example of nonconstant functions $f_1$ and $f_2$ which satisfy the condition $$\displaystyle\frac{f_1f_2'}{f_2^2}=\textrm{nonzero} \ \textrm{constant}=:c$$ is
$$f_1(x^1)=\frac{ck}{a}e^{ax^1}, \ \ f_2(x^1)=ke^{ax^1}, \ \ k,a\in \mathbb R\setminus \{0\}.$$
\end{remark}

\begin{corollary}
If $f_1=f_1(x^1)$ and $f_2=k_2\in \mathbb R\setminus \{0\}$, then, the nonzero vector field $V=\sum_{k=1}^2V^kE_k$ is a Ricci vector field satisfying $\nabla V=Q$ if and only if
$$(V^1,V^2)=(c_1,c_2)\in \mathbb R^2\setminus \{(0,0)\}.$$
\end{corollary}
\begin{proof}
It follows from Theorem \ref{t1}.
\end{proof}

\begin{example}
Let $f_1(x^1)=e^{x^1}$ and $f_2=k_2\in \mathbb R\setminus \{0\}$. Then,
$$V=e^{x^1}\displaystyle\frac{\partial }{\partial x^1}+k_2\displaystyle\frac{\partial }{\partial x^2},$$
is a Ricci vector field on $\left(\mathbb R^2, \ {g}=e^{-2x^1}dx^1\otimes dx^1+k_2^{-2}dx^2\otimes dx^2\right)$.
\end{example}

\begin{proposition}
If $f_1=f_1(x^2)$ and $f_2=k_2\in\mathbb R\setminus \{0\}$, then, the nonzero vector field $V=\sum_{k=1}^2V^kE_k$ with $V^i=V^i(x^j)$ for $i,j\in \{1,2\}$, $i\neq j$ is a Ricci vector field satisfying $\nabla V=Q$ if and only if $f_1=k_1\in\mathbb R\setminus \{0\}$ and
$$\left\{
    \begin{aligned}
&V^1=0\\
&V^2=c_2\in\mathbb R\setminus \{0\}
    \end{aligned}
  \right. \ \ \textrm{or} \ \
\left\{
    \begin{aligned}
&V^1=c_1\in\mathbb R\setminus \{0\}\\
&V^2=c_2\in\mathbb R
    \end{aligned}
  \right..
$$
\end{proposition}
\begin{proof}
In this case, \eqref{s1s} becomes
\begin{equation}\label{s1sbb}
\left\{
    \begin{aligned}
&\left(\displaystyle{\frac{f_1'}{f_1}}\right)'=\left(\displaystyle{\frac{f_1'}{f_1}}\right)^2\\
&f_1'V^2=0\\
&(V^2)'=-k_2\displaystyle{\frac{f_1'}{f_1^2}}V^1\\
&(V^1)'=0
    \end{aligned}
  \right. \ .
\end{equation}
From the 4th equation of the system follows that $V^1=c_1\in \mathbb R$, which, replaced in the 3rd equation, implies $(V^2)'=-c_1k_2\displaystyle{\frac{f_1'}{f_1^2}}$, which must be a constant, too, let's say, $c_0\in \mathbb R$. Therefore, $V^2(x^1)=c_0x^1+c_2$, $c_2\in \mathbb R$, and $-c_1k_2\displaystyle{\frac{f_1'}{f_1^2}}=c_0$, which, by integration, gives $c_1k_2\displaystyle{\frac{1}{f_1(x^2)}}=c_0x^2+c_3$, $c_3\in \mathbb R$.

(1) If $c_1=0$ (i.e., $V^1=0$), then, $c_0=c_3=0$ and $V^2=c_2\in \mathbb R\setminus \{0\}$, which implies, from the 2nd equation, that $f_1=k_1\in \mathbb R\setminus \{0\}$.

(2) If $c_1\neq 0$, then, $c_0x^2+c_3\neq 0$ for all $x^2\in \mathbb R$, therefore, $c_0=0$ and $c_3\in \mathbb R\setminus \{0\}$. In this case, $f_1=\displaystyle{\frac{c_1k_2}{c_3}}=:k_1\in \mathbb R\setminus \{0\}$ and $V^2=c_2\in \mathbb R$.
\end{proof}

\begin{example}
Let $f_1=k_1\in\mathbb R\setminus \{0\}$ and $f_2=k_2\in\mathbb R\setminus \{0\}$. Then,
$$V=k_1\displaystyle\frac{\partial }{\partial x^1}+k_2\displaystyle\frac{\partial }{\partial x^2}$$
is a Ricci vector field on $\left(\mathbb R^2, \ {g}=\displaystyle k_1^{-2}dx^1\otimes dx^1+k_2^{-2}dx^2\otimes dx^2\right)$.
\end{example}

\textit{Department of Mathematics}

\textit{Faculty of Mathematics and Computer Science}

\textit{West University of Timi\c{s}oara}

\textit{Bd. V. P\^{a}rvan 4, 300223, Timi\c{s}oara, Romania}

\textit{adarablaga@yahoo.com}

\end{document}